\documentclass[11pt]{amsart}
\usepackage[utf8]{inputenc}

\usepackage{amsmath,amsthm,amssymb,mathscinet,bbm}
\usepackage{parskip}
\usepackage{geometry}
\usepackage{xcolor} 
\usepackage{xfrac, nicefrac}
\usepackage{exscale, relsize}
\usepackage{xurl}

\usepackage{mathtools} 
\usepackage[hidelinks]{hyperref}
\author{Alexandros Kalogirou}
\address{Erd\H{o}s Center \\
R\'enyi Institute \\
Budapest, H-1053 \\}
\email{KALOGIRA@email.sc.edu}

\author{Andrew Lott}
\address{Department of Mathematics \\ University of Georgia \\ Athens, GA 30602}
\email{Andrew.Lott@uga.edu}

\author{\'Akos Magyar}
\address{Department of Mathematics \\ University of Georgia \\ Athens, GA 30602}
\email{amagyar@uga.edu}

\author{Akash Singha Roy}
\address{Charles University\\ Faculty of Mathematics and Physics \\ Department of Algebra\\ Sokolovsk\'a 83, 186 75 Praha 8\\ Czech Republic}
\email{akash01s.roy@gmail.com}

\subjclass[2020]{Primary 11B30; Secondary 11B05, 11L07, 11L15, 11P55}

\usepackage{accents}

\renewcommand{\pod}[1]{\allowbreak\mathchoice
  {\if@display \mkern 18mu\else \mkern 8mu\fi (#1)}
  {\if@display \mkern 18mu\else \mkern 8mu\fi (#1)}
  {\mkern4mu(#1)}
  {\mkern4mu(#1)}
}
\usepackage{graphicx}
\DeclareMathAlphabet{\curly}{U}{rsfs}{m}{n}

\newcommand\Z{\mathbb{Z}}
\allowdisplaybreaks

\usepackage{cleveref}
\crefname{section}{§}{§§}
\Crefname{section}{§}{§§}

\newcommand\NatNos{\mathbb N}

\newtheorem{thm}{Theorem}[section]
\newtheorem{cor}[thm]{Corollary}

\newtheorem{lem}[thm]{Lemma}

\theoremstyle{remark}
\newtheorem*{rmk}{Remark}

\newcommand\reals{\mathbb R}

\newcommand\Ree{\mathrm{Re}}
\newcommand\bbm{\mathbbm 1}

\newtheorem{innercustomgeneric}{\customgenericname}
\providecommand{\customgenericname}{}
\newcommand{\newcustomtheorem}[2]{
  \newenvironment{#1}[1]
  {
   \renewcommand\customgenericname{#2}
   \renewcommand\theinnercustomgeneric{##1}
   \innercustomgeneric
  }
  {\endinnercustomgeneric}
}

\newcustomtheorem{customthm}{Theorem}
\newcustomtheorem{customlemma}{Lemma}
\newcommand\sm\setminus

\newcommand\nuNhatalpha{\widehat \nu_N(\alpha)}

\newcommand\enalpha{e(\alpha n)}

\newcommand\Mqa{\mathcal M(q, a)}
\newcommand\totalMajorArc{\mathcal M}
\newcommand\OneAHat{\widehat\bbm_A}
\newcommand\OneAHatalpha{\OneAHat(\alpha)}

\newcommand\SkqaqPower{\left(\frac{S_k(q, a)}q\right)^s}

\newcommand\nukNhatalpha{\widehat \nu_{N}(\alpha)}
\newcommand\muNhatalpha{\widehat \mu_{N} (\alpha)}

\newcommand\Aks{\mathcal A_{k, s}}

\newcommand{\N}{\mathbb{N}}
\newcommand{\al}{\alpha}
\newcommand{\be}{\beta}

\newcommand{\eps}{\varepsilon}
\newcommand{\si}{\sigma}

\numberwithin{equation}{section}
\begin{document}
\title[Furstenberg--S\'ark\"ozy for sums of even number of odd powers]{The Furstenberg--S\'ark\"ozy theorem for sums of an even number of odd powers} 
\keywords{S\'ark\"ozy's theorem, van der Corput property, circle method}
\begin{abstract}
  We obtain a Furstenberg--S\'ark\"ozy-type result for sets $A\subset [N]$ whose difference set $A-A$ does not contain the sum of 
  $s$-many $k$-th powers of positive integers, with $k>1$ odd and $s>0$ even. Namely, we prove that such sets must satisfy a power-saving bound $|A| \, \ll \, N^{1-\frac1k\min\{s \, \sigma_k, \, 1/2\}+\eps}$ for any fixed $\eps>0$, where $\sigma_k >0 $ is any admissible saving in a classical one-variable Weyl estimate. In particular, we can take $\sigma_k=\max\left\{{2^{1-k}}, \, \frac{1}{k(k-1)}\right\}$ using the classical theory and the best currently available bounds for classical Weyl sums. A greedy construction produces a set $A\subset[N]$ with $|A|\gg N^{1-s/k}$ for which $A-A$ contains no sum of $s$-many positive $k$-th powers, so our power-saving bound is of the correct shape. 
\end{abstract}
\maketitle

\section{Introduction} 

A conjecture of Lov\'asz, proved independently by Furstenberg (using ergodic theory) and S\'ark\"ozy (using the circle method), asserts that if $A\subset[N]$ and $A-A$ does not contain any nonzero integer of the form $n^2$ with $n\in\Z$, then $|A|=o(N)$ \cite{Furstenberg77,Sarkozy78}. The strongest quantitative form currently known, due to Green and Sawhney \cite{GreenSawhney}, asserts that for any such $A \subset [N]$, we have 
\[
|A|\ll N\exp\bigl(-c\sqrt{\log N}\bigr)
\]
for some absolute constant $c>0$. Recently, Adajar, Agrawal, Choudhuri, Chuah, Fan, Hegde, Lott, Nandakumar, and Ponagandla \cite{IntersectiveGroup} combined the method of Green and Sawhney with the exponential sum estimates of Rice \cite{Rice19} to treat general \textit{intersective}\footnote{We say that $h \in \Z[T]$ is intersective if $\deg h \ge 2$ and $h(\NatNos) \cap q\Z \ne \emptyset$ for every $q \in \NatNos$.} polynomials. More precisely, they show that if $h\in\Z[T]$ is intersective of degree at least two, $A\subset[N]$, and $A-A$ does not contain any nonzero integer of the form $h(n)$ with $n\in\N$, then, for every fixed $d\in(0,1/2)$,
\[
|A|\ll_{h,d}N\exp\bigl(-c_{h}(\log N)^d\bigr).
\]
Doyle and Rice \cite{DoyleRice} obtained similar bounds for a broad class of intersective multivariable polynomials $P\in\Z[x_1,\dots,x_\ell]$ of degree at least two: if $A\subset[N]$ and $A-A$ does not contain any nonzero integer of the form $P(x_1,\dots,x_\ell)$ with $x_1,\dots,x_\ell\in\Z$, then
\[
|A|\ll_P N\exp\bigl(-c_P(\log N)^\mu\bigr)
\]
for some $\mu=\mu(\deg P,\ell)>0$. In addition, Fan and Lott \cite{fanlott} used the quantitative van der Corput method to prove that if $A\subset[N]$ and $A-A$ does not contain any nonzero integer of the form $x^2+y^2$ with $x,y\in\Z$, then
\[
|A|\ll_\epsilon N^{7/8+\epsilon}
\]
for every fixed $\epsilon>0$.

Our result is motivated by the developments above. Let $k>1$ be a fixed odd integer, $s>0$ be a fixed even integer, and let
\[
\Aks=\{a_1^k+ \dots + a_s^k:\ a_1, \dots, a_s\in \N\}.
\]
We study sets $A\subset [N]$ whose difference set $A-A$ contains no element of $\mathcal A_{k, s}$. Our purpose is to obtain a power saving bound of the form $|A|\ll_k N^{1-\tau_k}$ for some $\tau_k>0$. 
The exact value of the constant $\tau_k$ is dependent on the admissible exponent $\sigma_k>0$ in the classical Weyl estimate
\begin{equation}\label{1.1}
\sum_{m\leq X} e(\al m^k)\ll_{k,\delta} X^{1+\delta} \left(\frac{1}{q}+\frac{1}{X}+\frac{q}{X^{1/k}}\right)^{\sigma_k},\qquad (e(\al)=e^{2\pi i\al})
\end{equation}
for all $X \ge 0$, and for all $\alpha \in \reals$, $a \in \Z$ and $q \in \NatNos$ satisfying $\gcd(a, q)=1$ and $|\al-a/q|\leq q^{-2}$. Note that any exponent admissible in \eqref{1.1} necessarily satisfies $\si_k\leq 1/k$: Indeed, for any prime $p\nmid k$, taking $X=q=p^k$ and $\al=a/p^k$ gives
\[
\sum_{m\leq X} e(\al m^k)=p^{k-1}=X^{1-1/k}.
\]
Our first main result is the following.
\begin{thm}\label{thm:Sarkozyk} Fix any odd integer  $k>1$, and any positive even integer $s \in [k]$. For all $N\in\N$ and for any set $A\subset [N]$ satisfying $(A-A)\cap \Aks=\emptyset$, we have
\[
|A|\ll_{k,s, \epsilon}N^{1-\frac1k\min\{s \, \sigma_k, \, 1/2\}+\epsilon}\  \text{for any fixed}\ \epsilon >0.
\]
\end{thm}
The classical Weyl estimate allows us to take $\si_k = 2^{1-k}$ (see chapter 2 of \cite{Vau81}), while the best currently available bound in \eqref{1.1} \cite{Bour17} (for $k>5$) allows us to take  $\si_k=\frac{1}{k(k-1)}$.  
In particular, we can take $\si_3 = 1/4$ and $\si_5 = 1/16$, which by Theorem \ref{thm:Sarkozyk} lead to the bounds $|A| \,  \ll_\eps \, N^{5/6+\eps}$ for $k=3, s=2$, and $|A| \,  \ll_\eps \, N^{1-s/80+\eps}$ for $k=5$, $s \in \{2, 4\}$. For odd $k > 5$, Theorem \ref{thm:Sarkozyk} yields $|A| \, \ll_{k, \, s, \, \eps} \, N^{1-s/(k^2(k-1))+\eps}$ for any even $s \in [k-1]$. 

We now give a simple lower bound showing that the power-saving shape of the upper bound in Theorem \ref{thm:Sarkozyk} is sharp. Indeed, since any $n \in \Aks \cap [N]$ can be written as $n=a_1^k+\dots+a_s^k$ with $1\leq a_1,\dots,a_s\leq N^{1/k}$, we have the trivial bound
\[
|\Aks\cap[N]|\leq N^{s/k},
\]
so the greedy algorithm produces a set $A$ with $(A-A) \cap \Aks = \emptyset$ and 
\[
|A|
\gg N^{1-s/k}.
\]

For any odd integers $k_1, \dots, k_s \ge 1$, if a set avoids integers of the form $a_1^{k_1}+ \dots +a_s^{k_s}$ with $a_1, \dots, a_s \in \NatNos$, then it also avoids integers of the form $a_1^K + \dots +a_s^K$, where $K \coloneqq [k_1, \dots, k_s]$ is the least common multiple of $k_1, \dots, k_s$.  Hence, 
Theorem \ref{thm:Sarkozyk} admits the following generalization.  
\begin{thm}\label{thm:SarkozykDiff}
Fix any odd $k_1, \dots, k_s \in \NatNos$, at least one of which is not $1$. Let $K \coloneqq [k_1, \dots, k_s]$, and fix any positive even integer $s \le K$. Then for any set $A\subset [N]$ such that its difference set $A-A$ does not contain any integer of the form $a_1^{k_1}+\dots+a_s^{k_s}$, with $a_1, \dots, a_s \in\N$, we have   
$$|A| \, \ll_{s, \, k_1,  \,  \dots,  \, k_s,  \, \epsilon} ~ N^{1 \, - \, \frac1K\min\{s\sigma_K, \, 1/2\} \, + \, \epsilon}\text{ for any fixed }\epsilon>0.$$ 
\end{thm}
In \cite[Theorems 1.2 and 5.7]{Rice19}, Rice shows that for any intersective polynomials $h_1, \dots, h_j \in \Z[T]$, we have $|A| \ll N \exp(-c(\log N)^\theta)$ for any set $A \subset [N]$ for which $A-A$ avoids values of the form $\sum_{j=1}^s \, h_j(n_j)$ for all $n_1, \dots, n_j \in \NatNos$; here $c>0$ and $\theta \in (0, 1)$ are fixed. Theorem \ref{thm:SarkozykDiff} thus strengthens Rice's results for $h_j(T) = T^{k_j}$, with $s>0$ even and all $k_j$ odd.

Let us sketch the proof of Theorem \ref{thm:Sarkozyk}. Fix any real $c_1, \dots, c_s \ge 0$ such that $c_1+\dots+c_s = k-s$; as remarked at the end of section \ref{sec:MainThmCompletion}, our arguments go through and yield the same results for any choice of $c_1, \dots, c_s \ge 0$ for which $c_1+\dots+c_s \le k-s$, but having exact equality makes the computations cleaner. 
Define for any non-negative integer $n$, the weighted count 
\begin{equation}\label{eq:rksDef} 
r_{k, s}(n) \, = r_{k, s}(n; c_1, \dots, c_s) \, = \, \sum_{\substack{a_1, \dots, a_s \ge 1\\\ a_1^k+ \dots + a_s^k=n}} \, a_1^{c_1} \cdots a_s^{c_s};
\end{equation}
in particular, $r_{k, s}(0) = 0$. Consider the weight function $\mu_N: \Z \rightarrow \reals_{\ge 0}$ given by 
\begin{equation}\label{eq:muNDef}
\mu_N(d) \coloneqq r_{k, s}(|d|) \, e^{-|d|/N}.
\end{equation}
Even though $\mu_N$ may be non-zero outside $(-N, N)$, we still have $\mu_N(d)=0$ if $|d|$ cannot be written in the form $a_1^k+ \dots + a_s^k$ for any $a_1, \dots, a_s \in \NatNos$. 
Hence, if $(A-A) \cap \Aks = \emptyset$, then 
\begin{equation}\label{eq:nukNPlancherel}
\begin{split}
0 \, &= \, \sum_{m, \, n \,  \in \, A} \, \mu_N(m-n) \, = \sum_{\substack{m, \, n \,  \in \, A\\d \in \Z}} \, \mu_N(d) \, \int_0^1 \, e(\alpha (n-m+d)) \, d\alpha \,\\
&= \, \int_0^1 \, \left(\sum_{n \in A} \, e(\alpha n) \right) \, \left(\sum_{m \in A} \, e(-\alpha m) \right) \, \left(\sum_{d \in \Z} \, \mu_N(d) \, e(\alpha d) \right) \, d\alpha \, = \, \int_0^1 \, |\OneAHatalpha|^2 \, \muNhatalpha \, d\alpha;
\end{split}
\end{equation}
here we used the fact that $\sum_{d \in \Z} \, \mu_N(d)$ is absolutely convergent, since we can write 
\begin{equation}\label{eq:muNAbsConv}
\begin{split}
\sum_{d \ge 1} \, \mu_N(d) \, &= \, \sum_{d \ge 1}\, r_{k, s}(d) \, e^{-d/N} \,= \, \sum_{d \ge 1} \sum_{\substack{a_1, \dots, a_s \ge 1\\ a_1^k+ \dots + a_s^k=d}} \, a_1^{c_1} \cdots a_s^{c_s} \, e^{-d/N} \,\\ &= \, \sum_{\substack{a_1, \dots, a_s \ge 1}} \, a_1^{c_1} \cdots a_s^{c_s} \, \exp\left(-\frac{a_1^k+ \dots + a_s^k}N\right) \, = \, \prod_{j=1}^s \, \left(\sum_{a_j \ge 1} \, a_j^{c_j} \, \exp\left(-\frac{a_j^k}N\right)\right),
\end{split}
\end{equation}
and each factor in the last expression is easily seen to be absolutely convergent.

The main feature of the weight $\mu_N$ is revealed by its major arcs approximation. Writing $\alpha = a/q + \beta$ and $S_k(q,a)=\sum_{r\pmod{q}} e(ar^k/q)$ for $\alpha \in \reals$, $q \in \NatNos$ and $a \in \Z$, extensions of an estimate of Br\"{u}dern and Vaughan \cite{BV21} (see Lemmas \ref{lem:rkExpSum} and \ref{lem:MajorArcEstk} below) yield 
\begin{equation}\label{1.2}
\muNhatalpha =2 \,  \mathcal C_{k, s} \, \left(\frac{S_k(q,a)}{q}\right)^s 
\int_0^\infty \cos (2\pi \be w)\, e^{-w/N}\,dw\ +\ \text{error}
\end{equation}
for some constant $\mathcal C_{k, s}>0$. Both factors in the main term above  are non-negative: Indeed, as seen in the proof of Lemma \ref{lem:MajorArcEstk} below, we have
\begin{equation}\label{eq:CosineInt}
    \int_0^\infty \cos (2\pi \be w)\, e^{-w/N}\,dw = \frac N{1+(2\pi \beta N)^2} \, > \, 0,
\end{equation}
while the odd parity of $k$ implies that $S_k(q,a)$ is real-valued, since 
\begin{equation}\label{eq:SkqaReal}
\overline{S_k(q, a)} = \sum_{r \bmod q} \, e\left(-\frac{ar^k}q\right) = \sum_{r \bmod q} \, e\left(\frac{a(-r)^k}q\right) = \sum_{m \bmod q} \, e\left(\frac{am^k}q\right) = S_k(q, a), 
\end{equation}
which in turn implies that $S_k(q,a)^s \in \reals^+$ for even $s>0$.

As a consequence, the contribution of every major arc to the integral $\int_0^1 \, |\OneAHatalpha|^2 \, \muNhatalpha \,d\alpha$  
is not too negative; in fact, the total contribution of the major arcs is $O(|A| \, N^{1-1/(2k)+\eps})$.  Moreover, the  
major arc around $0$ contributes $\gg |A|^2$. 
Extending the classical Weyl estimate \eqref{1.1} (as in Lemma \ref{lem:SmoothWeyl} and Corollary \ref{cor:MinorArcBoundk}), we see that the contributions of the minor arcs 
is $O(|A| \, N^{1-s\, \si_k/k+\eps})$. 
Combining all these estimates with \eqref{eq:nukNPlancherel}, we obtain Theorem \ref{thm:Sarkozyk}. 

The positivity of the main term of \eqref{1.2} is the main reason the  
case of odd $k$ admits a relatively direct argument. For a single odd polynomial, the analogous complete sum is real but may be negative, and previous positive-exponential-sum constructions of Nin\v{c}evi\v{c}–Slijep\v{c}evi\v{c} \cite{nin} average over several dilations to overcome this problem. Fortunately, for sums of an even number of equal odd powers, the complete sum $S_k(q,a)$ is automatically raised to an even power, thus yielding the desired positivity. By contrast, for sums of two squares,  the quadratic complete sums may be negative or complex; the recent work of Fan and Lott \cite{fanlott} resolves this by averaging over dilations in the style of Nin\v{c}evi\v{c}–Slijep\v{c}evi\v{c} and then carefully choosing smooth weights which force the relevant complete sum and oscillatory integral to lie in a fixed sector, thereby yielding the positivity of the real part of their product. 

Theorem \ref{thm:Sarkozyk} requires the restriction $s \le k$ (which is reflected in all subsequent theorems as well). This restriction is imposed  mainly so that there exist non-negative $c_j$ satisfying $c_1+\dots+c_s \le k-s$. It should be possible to adapt our argument 
to $s \in (k, 2k)$, by allowing some of the $c_j$ to be negative, but any attempt at doing so would require  careful 
bookkeeping of several error terms, negative exponents and endpoint corrections in various integrals arising in our arguments. We omit the technical details, but we remark that a much simpler version of our argument can be used to provide the following $s$-independent power saving        
\begin{equation}\label{eq:UnifsSaving}
|A| \, \ll_{k, \, s, \, \epsilon} \, N^{1-\sigma_k/k+\epsilon},    
\end{equation}
for  
\textit{any} even $s>0$ and any $A \subset [N]$ satisfying $(A-A) \cap \Aks = \emptyset$. 
We discuss this in section \ref{sec:UnifsSavingProof}. 
It is worth noting that in the special case $k=3$, 
we can readily get a saving $|A| \, \ll \, N^{11/12+\epsilon}$  
\textit{without} a minor arc estimate, but only by carrying out a  
simple Farey dissection of the base interval all the way up to $N^{1/2}$. This phenomenon typically occurs in quadratic settings, hence it might be surprising that it works for sums of (an even number of) cubes as well. 

Our approach is equivalent to the quantitative van der Corput method. This idea goes back to Kamae-Mend\'{e}s France \cite{Kamae} and Ruzsa \cite{Ruzsa1984}; see also Montgomery \cite{mont94} and the general framework of Matolcsi and Ruzsa \cite{materuzsa} which is then applied to cubic residues in \cite{mat21}. A recent breakthrough due to Green \cite{green} applies this approach to obtain a power saving bound for sets whose difference set does not contain a shifted prime $p-1$. Although we formulate the proof directly, our argument in fact constructs a power-saving van der Corput witness function for the set $\mathcal A_{k, s}$. More precisely, with  
$\mathcal A_{k, s}(N) \coloneqq \mathcal A_{k, s} \cap [N]$, our proof produces non-negative coefficients $\lambda_d$, supported on $\mathcal A_{k, s}(N)$, satisfying 
\[
\sum_{d\in \mathcal A_{k, s}(N)} \lambda_d=1\ \ \text{and}\ \ 
\Ree\sum_{d\in \mathcal A_{k, s}(N)} \lambda_d\,e(d\al) \, \geq \, -C_{k, s, \eps} \, N^{-\tau_k+\eps}\quad(\al\in \mathbb{T}).
\]
Theorem \ref{thm:Sarkozyk} then follows by a short, standard Fourier analytic argument.

An alternative approach  
is to note that if $A \subset [N]$ satisfies $(A-A) \cap \Aks = \emptyset$, then $|A| \le E_{k, s}(N) \, + \, 1$, where $E_{k, s}(N)$ is the number of positive integers $n \le N$ which cannot be written as a sum of exactly $s$-many positive $k$-th powers. For large $s$, this approach seems more promising.  For instance, for $s \, \gg_k \, 1$, this automatically yields  $|A| \ll 1$. However, the problem of finding uniform bounds on $E_{k, s}(N)$ is open for general $k, s$. In the special case $k=3$, it is known that 
$$E_{3, s}(N) \, \ll \, \begin{cases}
    N^{0.89}, &\text{ if }s=4, \text{ (see \cite{KW10, W15})}\\
    N^{0.43}, &\text{ if }s=6, \text{ (see \cite{KW10, W15})}\\
    1, &\text{ for even }s \ge 8, \text{ (see \cite{Lin43, MC84})}\\
\end{cases}$$
which yields the same respective bounds on $|A|$ in the case $k=3$. 

`We remark on the situation for $s>k$ in Theorem \ref{thm:Sarkozyk}.  
\subsection*{Notation and conventions} In the rest of the manuscript $[a, b]$ just denotes the closed interval, with left endpoint $a$ and right endpoint $b$. Implied constants in $\ll$ and $O$-notation, 
and implicit constants in qualifiers like ``sufficiently large'', 
may depend on any parameters declared as ``fixed''. In particular, all implied constants  are always allowed to depend on the fixed parameters $k, s$,  $\epsilon$, and $c_1, \dots, c_s$. 
\section{The key major and minor arc estimates}
The first input in our arguments will be the following restated version of Theorem 4.1 in \cite{Vau81}.  
\begin{thm}\label{thm:WeylkMajorArcs}
Fix any $k \in \NatNos_{>1}$ and $\epsilon>0$, and define $S_k(q, a) \, \coloneqq \, \sum_{r \bmod q} \, e(ar^k/q)$. Then 
\begin{align}\allowdisplaybreaks
\sum_{n \le X}\, e\left(\alpha n^k\right) \, &= \, \frac{S_k(q, a)}q \int_0^X \, e(\beta t^k) \, dt + \, O(q^{1/2+\epsilon}(1+X^k|\beta|)^{1/2}) \label{eq:WeylkMajorArcsPreSub}\\
&= \, \frac{S_k(q, a)}{kq} \int_0^{X^k} \, \frac{e(\beta u)}{u^{1-1/k}} \, du \, + \, O(q^{1/2+\epsilon}(1+X^k|\beta|)^{1/2}), \label{eq:WeylkMajorArcsPostSub}
\end{align} 
uniformly in $X \ge 1$, in $q \in \NatNos$, in $a \in \Z$, and in $\alpha \in \reals$, where $\beta \coloneqq \alpha - a/q$.
\end{thm}
Note that \eqref{eq:WeylkMajorArcsPostSub} follows from \eqref{eq:WeylkMajorArcsPreSub} by taking $u \coloneqq t^k$. Moreover, even though the above result is typically stated with the condition $\gcd(a, q)=1$, it still holds without this condition because if $d \coloneqq \gcd(a, q)$, then $S_k(q, a) = dS_k(q/d, a/d)$ and $|q/d| \le |q|$.  

We first establish the following generalization of \cite[Lemma 2.5]{BV21}, which gives an explicit estimate on the exponential sum $\sum_{n \le T} \, r_{k, s}(n) \, \enalpha$, uniformly in all real phases $\alpha$. Note that the case $s=2$, $k=3$ and $c_1 = \dots = c_s = 0$ of the result below reduces to \cite[Lemma 2.5]{BV21}.
\begin{lem}\label{lem:rkExpSum} 
Fix any integer $k>1$, any $s \in \NatNos$, and any $\epsilon > 0$. 
Let $c_1, \dots, c_s \ge 0$ be constants such that $C \coloneqq c_1+\dots+c_s \, \le \, k-s$. Define $r_{k, s}(n)$ as in \eqref{eq:rksDef}, and set 
\begin{equation}\label{eq:CksDef}
\mathcal C_{k, s} \coloneqq \frac{\prod_{j=1}^s\Gamma((1+c_j)/k)}{k^s \, \Gamma((s+C)/k)}>0.
\end{equation}  
Then uniformly in all real $\alpha$ and $T \ge 1$, as well as in all integers $q \ge 1$ and $a$, we have  
\begin{equation}\label{eq:rkExpSum}
\sum_{n \le T} r_{k, s}(n) e(\alpha n) \, = \, \mathcal C_{k, s}\SkqaqPower \, \int_0^T \,  \frac{e(\beta w)}{w^{1-(s+C)/k}} \, dw \,  + \, O(T^{(s-1+C)/k} \, q^{1/2+\epsilon} \, (1+T|\beta|)^{1/2}),
\end{equation} 
where $\beta \coloneqq \alpha - a/q$. The implied constants are allowed to depend at most on $k, s, \epsilon, c_1, \dots, c_s$.  
\end{lem}
\begin{proof}
We will establish the lemma by inducting on $s$. The base case $s=1$ is an immediate consequence of the following general observation: For any constant $c \ge 0$ and for all $X \ge 0$, 
\begin{equation}\label{eq:rksExpSumBaseCase}
\sum_{m \le X} \, m^c \, e(\alpha m^k) \, = \, \frac{S_k(q, a)}{k q} \, \int_0^{X^k} \, \frac{e(\beta u)}{u^{1-(1+c)/k}} \, du \, + \, O_{k, c, \epsilon}(X^c \, q^{1/2+\epsilon} \, (1+X^k|\beta|)^{1/2}).   
\end{equation}
Indeed, by partial summation and Theorem \ref{thm:WeylkMajorArcs}, we see that 
\begin{align}
\sum_{m \le X} \, m^c \, e(\alpha m^k) \, &= \, X^c \sum_{m \le X} e(\alpha m^k) \, - \, \int_0^X \, \left(\sum_{m \le t} \, e(\alpha m^k) \right) \cdot c \, t^{c-1} \, dt \nonumber\\
&= \begin{multlined}[t] \label{eq:BaseCaseIntermed}
    \frac{S_k(q, a)}{kq} \left\{X^c \int_0^{X^k} \, \frac{e(\beta u)}{u^{1-1/k}} \, du \, - \, \int_0^X c \, t^{c-1} \, \int_0^{t^k} \, \frac{e(\beta u)}{u^{1-1/k}} \, du \, dt \right\}\\
    + \, O( X^c \, q^{1/2+\epsilon} \, (1+X^k|\beta|)^{1/2}), 
\end{multlined}
\end{align}
where we have noted that $\int_0^X \, t^{c-1} \cdot (1+t^k|\beta|)^{1/2} \, dt \, \le \, (1+X^k|\beta|)^{1/2}  \, \int_0^X \, t^{c-1} \, dt$. Now interchanging integrals, we see that the double integral in the main term of \eqref{eq:BaseCaseIntermed} equals 
$$\int_0^{X^k} \frac{e(\beta u)}{u^{1-1/k}} \, \int_{u^{1/k}}^X \, c \, t^{c-1}  \, dt \, du \, = \, X^c \, \int_0^{X^k} \frac{e(\beta u)}{u^{1-1/k}} \, du \, - \, \int_0^{X^k} \frac{e(\beta u)}{u^{1-(1+c)/k}} \, du.$$
Inserting this into \eqref{eq:BaseCaseIntermed} yields \eqref{eq:rksExpSumBaseCase}, establishing the base case $s=1$ of the lemma. 

Now assume the result is true for some $s \in \NatNos$; we will show that it also holds for $s+1$ in place of $s$. To this end, we start by writing   
\begin{align}\allowdisplaybreaks
\sum_{n \le T} r_{k, s+1}(n) \, e(\alpha n) \, &= \sum_{\substack{a_1, \dots, a_{s+1} \ge 1 \\ a_1^k+ \dots + a_{s+1}^k \le T}} \, a_1^{c_1} \cdots a_{s+1}^{c_{s+1}} \, e(\alpha (a_1^k+ \dots + a_{s+1}^k))\nonumber\\
&= \, \sum_{1 \le a_{s+1} \le T^{1/k}} \, a_{s+1}^{c_{s+1}} \, e(\alpha a_{s+1}^k) \, \sum_{\substack{a_1, \dots, a_{s} \ge 1\\a_1^k + \dots + a_{s}^k \le T-a_{s+1}^k}} \, a_{1}^{c_{1}} \, \cdots a_s^{c_s} \, e(\alpha (a_1^k + \dots + a_{s}^k)) \, \nonumber\\ &= \, \sum_{1 \le a_{s+1} \le T^{1/k}} \, a_{s+1}^{c_{s+1}}\, e(\alpha a_{s+1}^k) \, \sum_{m \le T-a_{s+1}^k} \, r_{k, s}(m)  \, e(\alpha m). \label{eq:rskInduct}
\end{align}
Applying the inductive hypothesis, we obtain 
\begin{multline*}
\sum_{n \le T} r_{k, s+1}(n) \, e(\alpha n) \, = \mathcal C_{k, s}\SkqaqPower \,  \sum_{1 \le a_{s+1} \le T^{1/k}} \, a_{s+1}^{c_{s+1}}\, e(\alpha a_{s+1}^k) \int_0^{T-a_{s+1}^k} \,  \frac{e(\beta w)}{w^{1-(s+C')/k}} \, dw \,\\  + \, O\left(T^{(s-1+C')/k} \, q^{1/2+\epsilon} \, (1+T|\beta|)^{1/2} \, \sum_{1 \le a_{s+1} \le T^{1/k}} \, a_{s+1}^{c_{s+1}}\right), 
\end{multline*}
where $C' \, \coloneqq \, c_1+\dots+c_s$, and where  
$$\mathcal C_{k, s} \coloneqq \frac{\prod_{j=1}^s\Gamma((1+c_j)/k)}{k^s \, \Gamma((s+C')/k)}.$$Interchanging the sum and the integral in the main term of the last estimate, and noting that the sum in the error term is at most $T^{c_{s+1}/k} \, \sum_{1 \le a_{s+1} \le T^{1/k}} \, 1 \, \le \, T^{(1+c_{s+1})/k}$, we obtain 
\begin{multline*}
\sum_{n \le T} r_{k, s+1}(n) \, e(\alpha n) \, = \mathcal C_{k, s}\SkqaqPower \,  \int_0^T \,  \frac{e(\beta w)}{w^{1-(s+C')/k}} \left(\sum_{1 \le a_{s+1} \le (T-w)^{1/k}} \, a_{s+1}^{c_{s+1}}\, e(\alpha a_{s+1}^k) \right) \, dw \,\\  + \, O\left(T^{(s+C)/k} \, q^{1/2+\epsilon} \, (1+T|\beta|)^{1/2}\right), 
\end{multline*}
where $C \coloneqq C'+c_{s+1} = c_1+\dots+ c_{s+1}$. We invoke \eqref{eq:rksExpSumBaseCase} to estimate the inner sum, and note that $ T^{c_{s+1}/k} \, q^{1/2+\epsilon} \, (1+T|\beta|)^{1/2} \, \int_0^T \, w^{(s+C')/k-1}\, dw \, \le \, T^{(s+C)/k} \, q^{1/2+\epsilon} \, (1+T|\beta|)^{1/2}$. This yields
\begin{equation}\label{eq:rks+1Prelim}
\sum_{n \le T} r_{k, s+1}(n) \, e(\alpha n) \, = \, \frac{\mathcal C_{k, s}}k \,  \left(\frac{S_k(q, a)}q\right)^{s+1}  
I_{k, s+1}(\beta, T)\\
+ O(T^{(s+C)/k} \, q^{1/2+\epsilon} \, (1+T|\beta|)^{1/2}),
\end{equation}
with $I_{k, s+1}(\beta, T) \coloneqq \int_0^T \, w^{(s+C')/k-1} \int_0^{T-w} \, {u^{(1+c_{s+1})/k - 1}} \, e(\beta(u+w)) \, du \, dw$. Setting $v \coloneqq u+w$ in the inner integral, then interchanging the resulting integral on $v$ with the outer integral on $w$, 
\begin{align}
I_{k, s+1}(\beta, T) \, &= \, \int_0^T \, e(\beta v) \, \int_0^v \, w^{(s+C')/k-1} \,  (v-w)^{(1+c_{s+1})/k - 1} dw \, dv \nonumber\\
&= \int_0^T \, e(\beta v) \, v^{(s+1+C)/k-1} \, dv \, \int_0^1 \,  z^{(s+C')/k-1} \,  (1-z)^{(1+c_{s+1})/k - 1} dz \\
&= \frac{\Gamma((s+C')/k) \, \Gamma((1+c_{s+1})/k)}{\Gamma((s+1+C)/k)} \, \int_0^T \, \frac{e(\beta v)}{v^{1-(s+1+C)/k}} \, dv \label{eq:Iks+1},
\end{align}
where we have set $w = vz$ and then used the standard identity relating the beta function with Gamma functions. Finally, inserting \eqref{eq:Iks+1} into \eqref{eq:rks+1Prelim}, and noting that 
$$\frac{\mathcal C_{k, s}}k \cdot \frac{\Gamma((s+C')/k) \, \Gamma((1+c_{s+1})/k)}{\Gamma((s+1+C)/k)} \, = \, \mathcal C_{k, s+1}$$ completes the induction step, and with it, the proof of Lemma \ref{lem:rkExpSum}. 
\end{proof}
In what follows, we write $A \ge O_+(B)$ (respectively, $A \ge -O_+(B)$) to mean that $A \ge \kappa B$ (respectively, $A \ge -\kappa B$) for some constant $\kappa>0$ depending at most on $k, s, c_1, \dots, c_s$ and $\epsilon$. 
We now employ Lemma \ref{lem:rkExpSum} to  obtain our main workhorse estimates for the major arcs.  
\begin{lem}\label{lem:MajorArcEstk}
Fix any odd integer $k>1$, any even integer $s>0$, and any $\epsilon>0$. Let $c_1, \dots, c_s \ge 0$ be constants such that 
$c_1+\dots+c_s \, = \, k-s$. Define $r_{k, s}(n)$ and $\mu_N(d)$ as in \eqref{eq:rksDef} and \eqref{eq:muNDef}. Then uniformly in all real $\alpha$, and in all integers $N \ge 1$, $q \ge 1$, and $a$, we have 
\begin{equation}\label{eq:MajorArcEstk}
\muNhatalpha \, = \, 2 \, \mathcal C_{k, s} \SkqaqPower \, 
\, \frac N{1+(2\pi \beta N)^2} \, + \, O(N^{1-1/k} \, q^{1/2+\epsilon} \, (1+N|\beta|)^{1/2})
\end{equation} 
where $\beta \coloneqq \alpha - a/q$. 
As a consequence, we have uniformly in all such $\beta, N, q, a$, 
\begin{equation}\label{eq:MajorArcBoundk}
 \muNhatalpha \, \ge \, - \, O_+(N^{1-1/k} \, q^{1/2+\epsilon} \, (1+N|\beta|)^{1/2}).   
\end{equation}
\end{lem}
\begin{proof}
We will first show that 
\begin{multline}\label{eq:MajorArcEstkPosHalf}
\sum_{n \ge 1} \, \mu_N(\alpha) \, e(\alpha n) \, = \, \mathcal C_{k, s} \SkqaqPower \, 
\int_0^\infty \, e(\beta w) \, e^{-w/N} \, dw \,   + \, O(N^{1-1/k} \, q^{1/2+\epsilon} \, (1+N|\beta|)^{1/2})
\end{multline}   
To this end, we use \eqref{eq:rkExpSum} to see that $e^{-t/N} \, \sum_{n \le t} \, r_{k, s}(n) \, e(\alpha n) \, \ll \, e^{-t/N} \, \big(q t (1+|\beta|) \big)^{O(1)}$ uniformly in $t \ge 1$. Hence, by partial summation and \eqref{eq:rkExpSum}, we obtain 
\begin{align}
\sum_{n \ge 1} \, \mu_N(n) \, e(\alpha n) \, &= \, \sum_{n \ge 1} \, r_{k, s}(n) \, e(\alpha n) \cdot e^{-n/N} \,  = \, \int_0^\infty \, \left(\sum_{n \le t} \, r_{k, s}(n) \, e(\alpha n)\right) \cdot \frac{e^{-t/N}}N \, dt \label{eq:FTPosHalfStep1}\\
&= \, 
\begin{multlined}[t] \label{eq:FTPosHalf}
   \mathcal C_{k, s}\SkqaqPower \, \int_0^\infty \, \int_0^t \,  e(\beta w) \, dw \cdot \frac{e^{-t/N}}N \, dt  \,\\ + \, O\left(\int_0^\infty \, t^{1-1/k} \, q^{1/2+\epsilon} \, (1+t|\beta|)^{1/2}  \cdot \frac{e^{-t/N}}N \, dt \right)
\end{multlined}
\end{align} 
Upon an interchange of integrals, the main term of \eqref{eq:FTPosHalf} equals
\begin{equation}\label{eq:FTPosHalfMainInt}
\begin{split}
 \mathcal C_{k, s}\SkqaqPower \, \int_0^\infty \, {e(\beta w)} \, \int_w^\infty \,  \frac{e^{-t/N}}N \, dt  \, dw \, = \, \mathcal C_{k, s}\SkqaqPower \, \int_0^\infty \, e(\beta w) \, e^{-w/N} \, dw.
\end{split}
\end{equation}
Next, taking $t \coloneqq Nu$, we see that the total error term in \eqref{eq:FTPosHalf} is
\begin{align} \allowdisplaybreaks
&\ll \, N^{1-1/k} \, q^{1/2+\epsilon} \int_0^\infty \, u^{1-1/k} \, (1+Nu|\beta|)^{1/2}  \, e^{-u} \, du, 
\label{eq:ErrPostSub}
\end{align}
which for $\beta \ne 0$ is   
\begin{align} \allowdisplaybreaks
&\ll \, N^{1-1/k} \, q^{1/2+\epsilon} \left\{\int_0^{1/(N|\beta|)} \, u^{1-1/k}  \, e^{-u} \, du \, + \, (N|\beta|)^{1/2} \, \int_{1/(N|\beta|)}^\infty \, u^{1-1/k+1/2}  \, e^{-u} \, du \right\} \nonumber\\
&\le \, N^{1-1/k} \, q^{1/2+\epsilon} \left\{\Gamma\left(2-\frac1k\right) \, + \, (N|\beta|)^{1/2} \, \Gamma\left(\frac52-\frac1k\right) \right\} \label{eq:FTPosHalfErrIntermed}\\
&\ll \, N^{1-1/k} \, q^{1/2+\epsilon} \, (1+(N|\beta|)^{1/2}) \, \ll  N^{1-1/k} \, q^{1/2+\epsilon} \, (1+N|\beta|)^{1/2}; 
\label{eq:FTPosHalfErr}
\end{align}
in the last line, we have used $(1+|\eta|)^{1/2} \, \asymp \, 1+|\eta|^{1/2}$ uniformly in all reals $\eta$, and observed that both the $\Gamma$-values in \eqref{eq:FTPosHalfErrIntermed} are  
$O_{k, s}(1)$ since $5/2-1/k > 2-1/k \ge 3/2$. Note that if on the other hand, we had $\beta = 0$, then the entire expression in \eqref{eq:ErrPostSub} is exactly  
$\Gamma(2-1/k) \, N^{1-1/k} \, q^{1/2+\epsilon}$, which is still (asymptotically) bounded by the expression in \eqref{eq:FTPosHalfErr}.   
Inserting \eqref{eq:FTPosHalfMainInt} and \eqref{eq:FTPosHalfErr} into \eqref{eq:FTPosHalf}, we obtain our desired estimate \eqref{eq:MajorArcEstkPosHalf}. 

Now since $\muNhatalpha \, = \, \sum_{n \in \Z} \, \mu_N(n) \, e(\alpha n) \, = \, 2 \, \Ree \sum_{n \ge 1} \, \mu_N(n) \, e(\alpha n)$, and since $S_k(q, a) \in \reals$, we see from \eqref{eq:MajorArcEstkPosHalf} that  
$$\muNhatalpha \, = \, 2 \, \mathcal C_{k, s} \SkqaqPower  \, \int_0^\infty \, \cos(2 \pi \beta w) \, e^{-w/N} \, dw \, + \, O(N^{1-1/k} \, q^{1/2+\epsilon} \, (1+N|\beta|)^{1/2}).$$
To establish \eqref{eq:MajorArcEstk}, it thus remains to show that the integral above is $N\big/({1+(2\pi \beta N)^2})$. Taking $w = Nu$ in the integral, this comes down to showing that 
\begin{equation*}
\text{ For any }\xi \in \reals, \text{ we have } \int_0^\infty \, \cos(\xi u) \, e^{-u} \, du \, = \, \frac1{1+\xi^2}. 
\end{equation*}
This follows easily from two applications of integration by parts, for we get  
\begin{align*}
\int_0^\infty \, \cos(\xi u) \, e^{-u} \, du \, = \, 1 - \xi\, \int_0^\infty \, \sin(\xi u) \, e^{-u} \, du \, = \, 1-\xi^2 \int_0^\infty \, \cos(\xi u) \, e^{-u} \, du.   
\end{align*}
This completes the proof of \eqref{eq:MajorArcEstk}. The second assertion \eqref{eq:MajorArcBoundk} follows immediately, since $S_k(q, a) \in \reals$ and $s$ is even, making $S_k(q, a)^s$ a nonnegative real number.  
\end{proof}
We conclude this section with our workhorse estimate for the minor arcs. Given $b \in \NatNos$, recall that $\sigma_b > 0$ was defined to be any constant such that for any fixed $\delta>0$, a Weyl estimate 
\begin{equation}\label{eq:Weyl}
\sum_{m \le X} \, e(\alpha m^b) \, \ll_{b, \delta} \, X^{1+\delta} \, \left(\frac1q \, + \, \frac1X \, + \frac q{X^b}\right)^{\sigma_b}  
\end{equation}
holds true for all $X \ge 0$,  $\alpha \in \reals$, $a \in \Z$ and $q \in \NatNos$  satisfying $\gcd(a, q)=1$ and $|\al-a/q|\leq q^{-2}$. As observed in the introduction, we have $\sigma_b \le 1/b$. 
\begin{lem}\label{lem:SmoothWeyl}
Fix any $b \in \NatNos$, any $c \ge 0$ and $\delta>0$. Uniformly in $\alpha \in \reals$ and $T \ge 1$, and uniformly in $a \in \Z$ and $q \in \NatNos$ satisfying $\gcd(a, q)=1$ and $|\alpha-a/q| \le 1/q^2$, 
we have 
\begin{equation}\label{eq:WeylTwokthPowers}
\sum_{m \ge 1} \, m^c \, e(\alpha m^b) \, \exp\left(-\frac{m^b}T \right) 
\, \ll_{b, \, s, \, c, \, \delta} \, T^{(1+c+\delta)/b} \, \left(\frac1q \, + \, \frac1{T^{1/b}} \, + \, \frac qT\right)^{\sigma_b}.
\end{equation}
\end{lem}
\begin{proof}
We start by observing that for any $m \in \NatNos$, $\rho \in [0, 1]$, and $A_1, \dots, A_m > 0$, we have
\begin{equation}\label{eq:PowerDistribIneq}
\frac{A_1^\rho \, + \, \dots \, + \, A_m^\rho}m \, \le \, (A_1 \, +\, \dots \, + \, A_m)^\rho \, \le \, A_1^\rho \, + \, \dots \, + \, A_m^\rho. 
\end{equation}
The left inequality follows from $(A_1 \, +\, \dots \, + \, A_m)^\rho \, \ge \, \max\{A_1^\rho, \dots, A_m^\rho\}$. To see the inequality on the right, we note that for any $\rho \in (0, 1)$ and $t>0$, the function $h_\rho(t) \, \coloneqq \, (1+t^\rho)-(1+t)^\rho$ has derivative $h_\rho'(t)=\rho t^{\rho-1}\cdot \big(1-(t/(t+1))^{1-\rho}\big) \, > \, 0$, so that $h_\rho$ is strictly increasing on the positive real line. As such, $h_\rho(t) >  h_\rho(0) = 0$ for all $t>0$, yielding $(1+t)^\rho \, \le 1+t^\rho$ for all such $t$ and for all $\rho \in [0, 1]$. Taking $t \coloneqq A_1/A_2$, we obtain $(A_1+A_2)^\rho \, \le \, A_1^\rho \, + \, A_2^\rho$ for any $A_1, A_2 > 0$ and any $\rho \in [0, 1]$. Inducting on this establishes the second inequality in \eqref{eq:PowerDistribIneq}.

Now with $\psi_T(t) \coloneqq t^c \, \exp(-t^b/T)$, we apply partial summation as in \eqref{eq:FTPosHalfStep1} to obtain 
\begin{align}
 \sum_{m \ge 1} \, m^c \, e(\alpha m^b) \, \exp\left(-\frac{m^b}T \right) \, = \, \sum_{m \ge 1} \, e(\alpha m^b) \, \psi_T(m) \, = \, -\int_0^\infty \, \left(\sum_{m \le t} \, e(\alpha m^b)\right) \, \psi_T'(t) \, dt.   
\end{align}
Inserting \eqref{eq:Weyl} and invoking  \eqref{eq:PowerDistribIneq}, we find that
\begin{equation}\label{eq:PostWeylSplit}
\sum_{m \ge 1} \, m^c \, e(\alpha m^b) \, \exp\left(-\frac{m^b}T \right) \, \ll \, q^{-\sigma_b} \, J_{1+\delta}(T) \, + \, J_{1+\delta-\sigma_b}(T) \, + \, q^{\sigma_b} \, J_{1+\delta-b \, \sigma_b}(T),
\end{equation}
where $J_\rho(T) \coloneqq \int_0^\infty \, t^\rho \, |\psi_T'(t)| \, dt$, and where (in the entire argument) our implied constants depend at most on $b, s, c$, $\delta$ and the parameter $\rho$ below. We claim that for any fixed $\rho>0$, we have
\begin{equation}\label{eq:JrhoBound}
J_\rho(T) \, \ll \, T^{(\rho+c)/b} 
\end{equation}
uniformly in $T \ge 1$. Indeed, since $\psi_T'(t) \, = \, \exp(-t^b/T) \, t^{c-1} \, (c-b \, t^b/T)$, we see that  $0 < \psi_T'(t) \, \le \, c \, t^{c-1} \, \exp(-t^b/T)$ for $0 < t \le (Tc/b)^{1/b}$, whereas $\psi_T'(t)<0$ for $t>  (Tc/b)^{1/b}$. Hence
\begin{align*}\allowdisplaybreaks
J_\rho(T) \, \le \, c\int_0^{(Tc/b)^{1/b}} \, t^{\rho+c-1} \, \exp\left(-\frac{t^b}T\right) \, dt \, + \, \int_{(Tc/b)^{1/b}}^\infty \, t^\rho \, (-\psi_T'(t)) \, dt.    
\end{align*} 
Integrating the second integral by parts and noting that $\psi_T((Tc/b)^{1/b}) \, \asymp \, T^{c/b}$, we obtain 
\begin{align*}
J_\rho(T) \, &\ll \, \int_0^{(Tc/b)^{1/b}} \, t^{\rho+c-1} \, \exp\left(-\frac{t^b}T\right) \, dt \, + \, T^{(\rho+c)/b} \, + \, \int_{(Tc/b)^{1/b}}^\infty \, \rho \, t^{\rho-1} \, \psi_T(t) \, dt\\
&\ll \,  T^{(\rho+c)/b} \, + \, \int_0^\infty \, t^{\rho+c-1} \, \exp\left(-\frac{t^b}T\right) \, dt \, \ll \,  T^{(\rho+c)/b} \, + \, T^{(\rho+c)/b} \,  \int_0^\infty \, u^{(\rho+c)/b-1} \, e^{-u} \, du, 
\end{align*}
where in the last step above, we have set $u \coloneqq t^b/T$. Since the last integral in the above display is $\Gamma((\rho+c)/b) \, \ll \, 1$ (as $\rho>0$, $c \ge 0$), we obtain the claimed bound  \eqref{eq:JrhoBound}. 

Finally, applying \eqref{eq:JrhoBound} with $\rho \in \{1+\delta, \, 1+\delta - \sigma_b \, 1+\delta - b \sigma_b\}$ (which are all strictly positive as  $\sigma_b \le 1/b$), we deduce from \eqref{eq:PostWeylSplit} that 
$$\sum_{m \ge 1} \, m^c \, e(\alpha m^b) \, \exp\left(-\frac{m^b}T \right) \, \ll \, T^{(1+c+\delta)/b} \, \left(\frac1{q^{\sigma_b}} \, + \, \frac1{T^{\sigma_b/b}} \, + \, \left(\frac qT\right)^{\sigma_b}\right),$$
which by a final application of \eqref{eq:PowerDistribIneq}, completes the proof of Lemma \ref{lem:SmoothWeyl}. 
\end{proof}
The following consequence of Lemma \ref{lem:SmoothWeyl} is what we will need to bound the total contribution of the minor arcs; the amplification of the power saving is the key to our main results. 
\begin{cor}\label{cor:MinorArcBoundk}
Fix $c_1, \dots, c_s \ge 0$ 
such that $c_1+\dots+c_s = k-s$. For any fixed $\delta>0$, we have 
\begin{equation}\label{eq:MinorArcEstk} 
\muNhatalpha \, \ll \, 
N^{1+\delta/k} \, \left(\frac1q \, + \, \frac1{N^{1/k}} \, + \, \frac qN\right)^{s\, \sigma_k},
\end{equation}
uniformly in all $N \in \NatNos$, $\alpha \in \reals$, $q \in \NatNos$ and $a \in \Z$ satisfying $\gcd(a, q)=1$ and $|\alpha-a/q| \le 1/q^2$.
\end{cor}
\begin{proof}
Since $\muNhatalpha \, = \, 2 \, \Ree \, \sum_{n \ge 1} \, \mu_N(n) \, e(\alpha n)$, it suffices to prove the claimed bound on this latter sum. Now proceeding exactly as in \eqref{eq:muNAbsConv}, we see that 
\begin{align}\allowdisplaybreaks
\sum_{n \ge 1} \, \mu_N(n) \, e(\alpha n) \, 
&= \, \sum_{\substack{a_1, \dots, a_s \ge 1}} \, a_1^{c_1} \cdots a_s^{c_s} \, e\big(\alpha (a_1^k+ \dots + a_s^k)\big) \, \exp\left(-\frac{a_1^k+ \dots + a_s^k}N\right)\\  
&= \, \prod_{j=1}^s \, \left(\sum_{a_j \ge 1} \, a_j^{c_j} \, e(\alpha a_j^k) \, \exp\left(-\frac{a_j^k}N\right)\right).  \label{eq:muNPosHalfProd}
\end{align}
The corollary now follows by bounding each of the $s$ factors in the last expression above by means of Lemma \ref{lem:SmoothWeyl} (with, say, ``$\delta/s$'' playing the role of ``$\delta$'' in Lemma \ref{lem:SmoothWeyl}).
\end{proof}
\begin{rmk}
As an attempt at an alternative to Lemma  \ref{lem:SmoothWeyl} and Corollary \ref{cor:MinorArcBoundk}, one may try to bound $\muNhatalpha$ by carrying out an ``induction-plus-partial summation'' argument analogous to Lemmas \ref{lem:rkExpSum} and  \ref{lem:MajorArcEstk}. However, this leads to a much worse bound than \eqref{eq:MinorArcEstk}, namely
$$\muNhatalpha \, \ll \, N^{1+\delta/k} \, \left(\frac1q \, + \, \frac1{N^{1/k}} \, + \, \frac qN\right)^{\sigma_k},$$
which by the argument in the next section,  
yields $|A| \, \ll \, N^{1-\sigma_k/k+\epsilon}$ for {any} even $s \in [k]$. Not only is this much weaker than Theorem \ref{thm:Sarkozyk}, but also it does not reflect the expected decay of the bound with $s$. In particular, 
for $s$ exceeding the Waring number, we have $|A| \, \ll \, 1$, rendering the previous bound redundant. 

On the other hand, the ``induction-plus-partial summation'' approach of Lemmas \ref{lem:rkExpSum} and  \ref{lem:MajorArcEstk} is very well-suited to the major arcs estimate. If on the other hand, we were to estimate the sum $\sum_{m \ge 1} \,  m^c \, e(\alpha m^k) \, \exp\left(-{m^k}/N \right)$ by partial summation in conjunction with Theorem \ref{thm:WeylkMajorArcs}, then we would get the following ``major arcs'' analogue of Lemma \ref{lem:SmoothWeyl} for any fixed $c \ge 0$.  
$$\sum_{m \ge 1} \, m^c \, e(\alpha m^k) \, \exp\left(-\frac{m^k}N \right) \, = \, \frac{S_k(q, a)}{kq} \, \int_0^\infty \, \frac{e(\beta u) \, e^{-u/N}}{u^{1-(1+c)/k}} \, du \, + \, O(N^{c/k} \, q^{1/2+\epsilon} \, (1+N|\beta|)^{1/2}).$$
But if we use this estimate on each of the factors in 
\eqref{eq:muNPosHalfProd}, then the total error upon expanding out the product becomes too large, giving a much worse estimate than \eqref{eq:MajorArcEstk}.   
\end{rmk}
\section{Completing the proof of Theorem \ref{thm:Sarkozyk}} \label{sec:MainThmCompletion}
\subsection{The major arcs} 
We always assume that $N$ is any positive integer, sufficiently large in terms of $k$. In the spirit of the classical circle method, we define the major arcs by 
\begin{equation}\label{eq:MajorArcsDef}
\Mqa \coloneqq \left\{\alpha \in [0, 1]: \,
\left|\left|\alpha-\frac aq\right|\right|_{\mathbb T} 
\, \le \,  \frac1{qN^{1-1/k}}\right\}, ~~ 1 \le q \le N^{1/k}, ~~ 0 \le a <  q,~~\gcd(a, q)=1,
\end{equation}
where $||x||_{\mathbb T} \coloneqq \min\limits_{r \in \Z} |x-r| = \min\{x, 1-x\}$ for $x \in [0, 1]$. Note that $$\mathcal M(1, 0) = \left[0, \, \frac1{N^{1-1/k}}\right] \, \cup \, \left[1-\frac1{N^{1-1/k}}, \, 1\right],$$ whereas 
$\mathcal M(q, a) = \{\alpha \in [0, 1]: |\alpha-a/q| \le 1/(qN^{1-1/k})\}$ for all the other $q, a$ in \eqref{eq:MajorArcsDef}.

We claim that the $\Mqa$ defined in \eqref{eq:MajorArcsDef} are all pairwise disjoint. Indeed, if $\alpha \in 
\Mqa \cap \mathcal M(q', a')$ for any $q, a, q', a'$ satisfying $2 \le q, \, q' \le N^{1/k}$, $0 \le a < q$, $0 \le a' < q'$, $\gcd(a, q) = \gcd(a', q')=1$, and $a/q \ne a'/q'$, then by the triangle inequality, 
$$\frac1{N^{1-1/k}} \left(\frac1q + \frac1{q'}\right) \ge \left|\alpha - \frac aq\right| + \left|\alpha - \frac{a'}{q'}\right| \ge \left|\frac aq - \frac{a'}{q'}\right| = \frac{|a'q-aq'|}{qq'} \ge \frac1{qq'},$$
leading to $N^{1-1/k} \le q+q' \le 2N^{1/k}$, which is false for all $N \gg 1$. Moreover, for any $\alpha \in \Mqa$ with $2 \le q \le N^{1/k}$, $0 \le a < q$ and $\gcd(a, q)=1$, we must have $1 \le a \le q-1$, so that 
$$\alpha \, \ge \, \frac aq- \frac1{qN^{1-1/k}} \, \ge \, \frac1q\left(1-\frac1{N^{1-1/k}}\right) \, \ge \, \frac1{N^{1/k}}\left(1-\frac1{N^{1-1/k}}\right) \, > \, \frac1{N^{1-1/k}}$$
(where the last inequality recalls that $k \ge 3$ and that $N \gg_k \, 1$), and 
$$\alpha \, \le \, \frac aq + \frac1{qN^{1-1/k}} \, \le \, \frac{q-1}q + \frac1{qN^{1-1/k}} \, = \, 1-\frac1q\left(1-\frac1{N^{1-1/k}}\right) \, < \, 1-\frac1{N^{1-1/k}},$$
showing that $\alpha \not\in \mathcal M(1, 0)$. Consequently, the union
$$\totalMajorArc \coloneqq \bigcup_{1 \le q \le N^{1/k}} \, \bigcup_{\substack{0 \le a < q\\\gcd(a, q)=1}} \Mqa$$
is genuinely a \textit{disjoint} union, so that 
\begin{equation}\label{eq:MajorArcFirstSplit}
\int_{\totalMajorArc} \, |\OneAHatalpha|^2 \, \muNhatalpha \, d\alpha \, = \, 
\sum_{1 \le q \le N^{1/k}} \, \sum_{\substack{0 \le a < q\\\gcd(a, q)=1}} ~~ \int\limits_{\mathcal M(q, a)} ~ |\OneAHatalpha|^2 \, \muNhatalpha \, d\alpha.
\end{equation}
For all $\alpha \in \mathcal M(q, a)$, 
we have $|\alpha-a/q| \, \le \, 1/(q \, N^{1-1/k})$ and $1 \le q \le N^{1/k}$. (For $\alpha \in \mathcal M(1, 0)$, we have this with $q=1$ and with $a \in \{0, 1\}$.) As such, \eqref{eq:MajorArcBoundk} yields
\begin{equation}\label{eq:muNHatUnifErrBound}
\begin{split}
\muNhatalpha \, &\ge \, - \, O_+\left(N^{1-1/k} \, q^{1/2+\epsilon} \, \left(1+ \frac N{q N^{1-1/k}}\right)^{1/2}\right) \, \ge \, -O_+\left(N^{1-1/k} \, q^{\epsilon} \, \left(q+ N^{1/k}\right)^{1/2}\right)\\ &\ge \,  -O_+\left(N^{1-1/k} \, (N^{1/k})^{\epsilon} \, \left(N^{1/k}+ N^{1/k}\right)^{1/2}\right) \, \ge \, -O_+(N^{1 \, + \, (-0.5 \, + \, \epsilon)/k}).
\end{split}
\end{equation}
uniformly in all $\alpha \in \mathcal M$.  
We use the above bound for all $\alpha \in \totalMajorArc \, \sm \, [0, N^{-1-\epsilon/k}]$, and then observe that $\int_{\totalMajorArc \, \sm \, [0, N^{-1-\epsilon/k}]} \, |\OneAHatalpha|^2 \, d\alpha$ is at most 
\begin{equation}\label{eq:SecondMomentSize}
\int_0^1 \, |\OneAHatalpha|^2  \, d\alpha \, = \, \int_0^1 \, \sum_{m, n \in A} \,  e((m-n)\alpha) \, d\alpha \, = \,  \sum_{m, n \in A} \, \int_0^1 \, e((m-n)\alpha) \, d\alpha \, = \, \sum_{m \in A} \, 1 \, = \, |A|,
\end{equation} 
Doing this, and noting that $(0, N^{-1-\epsilon/k}] \subset \mathcal M(1, 0)$, we obtain from \eqref{eq:MajorArcFirstSplit},  
\begin{equation}\label{eq:MajorArcSecondSplit}
\int_{\totalMajorArc} \, |\OneAHatalpha|^2 \, \muNhatalpha \, d\alpha \, \ge \, \int_0^{N^{-1-\epsilon/k}} \, |\OneAHatalpha|^2 \, \muNhatalpha \, d\alpha \, - \, O_+(N^{1 \, + \, (-0.5 \, + \, \epsilon)/k} \, |A|). 
\end{equation}
To estimate the first integral on the right above, note that uniformly in all $\alpha \in (0, N^{-1-\epsilon/k}]$,  
\begin{align}\label{eq:OneAHatDev}
|\OneAHatalpha|^2 &= \sum_{m, n \in A} e((m-n)\alpha) \, = \, \sum_{m ,n \in A} \, \{1+O(|m-n|\alpha)\} \nonumber\\ 
&= \, |A|^2 + O(|A|^2 N \alpha) \, \ge \, |A|^2 \, - \, O_+(|A|^2 \, N^{-\epsilon/k}) \, \ge \, |A|^2/2. 
\end{align}
Moreover, applying \eqref{eq:MajorArcEstk} to $q=1$ and $a = 0$, we have uniformly in all $\alpha \in (0, N^{-1-\epsilon/k}]$, 
\begin{align}\label{eq:nukNHatCloseto0}\allowdisplaybreaks
\muNhatalpha \, &= \, \frac{2 \, \mathcal C_{k, s} \, N}{1+(2\pi \alpha N)^2} \, + \, O(N^{1-1/k} \, (1+N|\alpha|)^{1/2}) \, \gg \, N, 
\end{align}
where in the last bound, we have noted that $0 \le N\alpha \le N^{-\epsilon/k}$. Inserting   
\eqref{eq:OneAHatDev} and \eqref{eq:nukNHatCloseto0}  into \eqref{eq:MajorArcSecondSplit}, we obtain the following clean lower bound for the total contribution of all our major arcs
\begin{equation}\label{eq:MajorArcFinal}
\int_{\totalMajorArc} \, |\OneAHatalpha|^2 \, \muNhatalpha \, d\alpha \, \gg  \, |A|^2 \cdot N^{- \epsilon/k} \, - \, O_+(N^{1 \, + \, (-0.5 \, + \, \epsilon)/k} \, |A|).
\end{equation}
\subsection{The minor arcs}
For any $\alpha \in [0, 1]$, Dirichlet's approximation theorem shows that there exist $a, q \in \Z$ with $1 \le q \le N^{1-1/k}$, $0 \le a < q$, $\gcd(a, q)=1$, and $|\alpha-a/q| \le 1/(qN^{1-1/k})$. If $\alpha \, \in \, [0, 1] \sm \totalMajorArc$, then we must have $q>N^{1/k}$. Hence $q \in (N^{1/k}, N^{1-1/k}]$, so that \eqref{eq:MinorArcEstk} yields
$$\muNhatalpha \, \ll \,  
N^{1+\epsilon/k} \, \left(\frac1q \, + \, \frac1{N^{1/k}} \, + \, \frac qN\right)^{s\, \sigma_k} 
\, 
\ll \, N^{1 \, + \, (-s \, \sigma_k  +\epsilon)/k}$$
uniformly in $\alpha \in [0, 1] \sm \totalMajorArc$, with $\sigma_k$ as defined before  
Theorem \ref{thm:Sarkozyk}. Consequently, by \eqref{eq:SecondMomentSize}, 
\begin{equation}\label{eq:MinorArcFinal}
\int_{[0, 1] \sm \totalMajorArc} \, |\OneAHatalpha|^2 \, \muNhatalpha \, d\alpha \, \ge \, -O_+( N^{1 \, + \, (-s \, \sigma_k  +\epsilon)/k} \, |A|).
\end{equation}
Finally, if $(A-A) \cap \Aks = \emptyset$, then $\int_0^1 \, |\OneAHatalpha|^2 \muNhatalpha \, d\alpha = 0$ by 
\eqref{eq:nukNPlancherel}. 
Hence, from \eqref{eq:MajorArcFinal} and \eqref{eq:MinorArcFinal},
$$|A| \, \ll \, N^{\epsilon/k} \, \left(N^{1 \, + \, (-0.5+\epsilon)/k} \, + \, N^{1 \, + \, (-s \, \sigma_k  +\epsilon)/k}\right) \, \ll \, N^{1 \, - \, \frac1k\min\{s\, \sigma_k, \, 1/2\}+\epsilon}.$$ 
This concludes the proof of Theorem \ref{thm:Sarkozyk}. 
\hfill \qedsymbol

\begin{rmk}
As alluded to in the introduction, our arguments go through for any $c_1, \dots, c_s \ge 0$ for which $C \coloneqq c_1 + \dots + c_s \le k-s$. Indeed, using Lemma \ref{lem:rkExpSum} directly, the same argument as given for \eqref{eq:MajorArcEstkPosHalf} yields the following analogue of \eqref{eq:MajorArcEstk},  uniformly in all $\alpha \in \reals$, $q \in \NatNos$, $a \in \Z$. 
\begin{equation}\label{eq:GenMajorArcEstk} 
\muNhatalpha \, = \, 2 \, \mathcal C_{k, s} \, \SkqaqPower \, 
\int_0^\infty \, \frac{\cos(2 \pi \beta  w)}{w^{1-(s+C)/k}} \, e^{-w/N} \, dw \,  + \, O\left(N^{(s-1+C)/k} \, q^{1/2+\epsilon} \, \left(1+N|\beta|\right)^{1/2}\right).
\end{equation}
To show that the main term here is non-negative (and hence obtain the corresponding analogue of \eqref{eq:MajorArcBoundk}), it suffices to take $w \coloneqq Nu$ in the above integral, and use the general fact that
\begin{equation}\label{eq:Laplace}
\text{For all }\rho \in (0, 1]\text{ and }\xi \in \reals, \text{ we have }\int_0^\infty \, \cos(\xi u) \, u^{\rho-1} \, e^{-u} \, du \, = \, \frac{\Gamma(\rho) \, \cos(\rho \arctan \xi)}{(1+\xi^2)^{\rho/2}} \, > \, 0.
\end{equation}
Here the equality follows by writing $\cos(\xi u) = \Ree(e^{i \, \xi u})$ and using the Laplace--transform identity $\int_0^\infty \, u^{\rho-1} \, e^{-zu} \, du \, = \, \Gamma(\rho) \, z^{-\rho}$ which holds for any $\rho>0$ and any $z$ with $\Ree(z)>0$. Note also that we are applying \eqref{eq:Laplace} for $\rho \coloneqq (s+C)/k$, which lies in $(0, 1]$ precisely because $C \le k-s$.

The only remaining complication in this generality is to obtain the analogue of \eqref{eq:nukNHatCloseto0} uniformly for $\alpha \in (0, N^{-1-\epsilon/k}]$. This can be done by invoking \eqref{eq:GenMajorArcEstk} for $q=1$ and $a=0$, to find that 
    \begin{align}\allowdisplaybreaks
\muNhatalpha \, 
&\gg \,  
\int_0^{N^{1+\epsilon/(2k)}} \, \frac{\cos(2 \pi \alpha w)}{w^{1-(s+C)/k}} \, e^{-w/N} \, dw \,  - \, O_+\left(N^{(s-1+C)/k} \, + \, \int_{N^{1+\epsilon/(2k)}}^\infty \, e^{-w/N} \, w^{(s+C)/k-1} \, dw \, \right)\nonumber\\
&\gg \, 
\int_0^{N^{1+\epsilon/(2k)}} \, e^{-w/N} \, {w^{(s+C)/k-1}} \, dw \,  - \, O_+\left(N^{(s-1+C)/k} \, + \, \int_{N^{1+\epsilon/(2k)}}^\infty \, e^{-w/N} \, w^{(s+C)/k-1} \, dw \, \right), \label{eq:muNhaCloseto0GenIntermed}
\end{align}
where we have noted that $\cos(2 \pi \alpha w) = 1+O(\alpha^2 w^2) = 1+O(N^{-\epsilon/k}) \ge 1/2$. Setting $v \coloneqq w/N$, we see that the integral in the error term of \eqref{eq:muNhaCloseto0GenIntermed} is  
\begin{align}
N^{(s+C)/k} \, \int_{N^{\epsilon/(2k)}}^\infty \, e^{-v} \, v^{(s+C)/k-1} \, dv \, &\le \, N^{(s+C)/k} \, \exp\left(-\frac12 \, N^{\epsilon/(2k)}\right) \, \int_{N^{\epsilon/(2k)}}^\infty \, e^{-v/2} \, v^{(s+C)/k-1} \, dv \nonumber\\ &\ll 
\, \frac1{N^2} \, \int_0^\infty \, e^{-u} \, u^{(s+C)/k-1} \, du \le \,  
\frac1{N^2} \, \Gamma\left(\frac{s+C}k\right) \, \ll \, \frac1{N^2}, \label{eq:GammaTail}
\end{align}
while the integral in the main term of \eqref{eq:muNhaCloseto0GenIntermed} is 
\begin{multline}\label{eq:GammaHead}
N^{(s+C)/k} \left\{\int_0^\infty \, e^{-v} \, v^{(s+C)/k-1} \, dv \, - \, 
\int_{N^{\epsilon/(2k)}}^\infty \, e^{-v} \, v^{(s+C)/k-1} \, dv \right\}\\ = \, N^{(s+C)/k} \, \Gamma\left(\frac{s+C}k\right) \, + \, O\left(\frac1N\right) \, \gg \, N^{(s+C)/k},
\end{multline}
where we have used \eqref{eq:GammaTail} and $(s+C)/k \le 1$ in the equality above. Inserting \eqref{eq:GammaHead} and \eqref{eq:GammaTail} into \eqref{eq:muNhaCloseto0GenIntermed}, we obtain the desired analogue of \eqref{eq:nukNHatCloseto0}, which is 
\begin{equation}
\muNhatalpha \, \gg \, N^{(s+C)/k},\text{ uniformly in all  $\alpha \in (0, N^{-1-\epsilon/k}]$.}  
\end{equation}
The analogue of Corollary \ref{cor:MinorArcBoundk} is $\muNhatalpha \, \ll \, N^{(s+C+\delta)/k} \, \left(1/q \, + \, 1/{N^{1/k}} \, + \, q/N\right)^{s\, \sigma_k}$, and  follows from Lemma \ref{lem:SmoothWeyl} by the same argument. The rest of the proof goes through, with all the main terms and error terms scaled by the same power of $N$, and we again end up with Theorem \ref{thm:Sarkozyk}.
\end{rmk}
\section{The $s$-independent estimate \eqref{eq:UnifsSaving}}\label{sec:UnifsSavingProof}
In this section, we summarize how the arguments given for Theorem \ref{thm:Sarkozyk} can be adapted to establish the bound $|A| \ll_{k, \, s, \, \epsilon} \, N^{1-\sigma_k/k+\epsilon}$ for \textit{any} even $s>0$. We consider the weight function 
\begin{equation}\label{eq:nuNDef}
\nu_N(d) \, \coloneqq \, \bbm_{d \in (-N, N)} ~ R_{k, s}(|d|) \, |d|^{1-s/k} \, \left(1-\frac{|d|}N\right),~~~\text{ where }~~~R_{k, s}(n) \, \coloneqq \, \sum_{\substack{a_1, \dots, a_s \ge 1\\a_1^k+\dots+a_s^k = n}} \, 1;
\end{equation}
hence $\nu_N$ and $R_{k, s}$ play the role of $\mu_N$ and $r_{k, s}$ from the previous sections (and in fact, $R_{k, s}$ is the same as $r_{k, s}$ if we take $c_1 = \dots = c_s = 0$). Arguments entirely similar to (and much simpler than) those given for Lemmas \ref{lem:rkExpSum} and  \ref{lem:MajorArcEstk} yield their following analogues,  uniformly in all $T \ge 1$, $N \in \NatNos$, $q \in \NatNos$, $a \in \Z$ and $\alpha \in \reals$, with $\beta \coloneqq \alpha-a/q$ and $C_{k, s} \coloneqq \Gamma(1+1/k)^s/\Gamma(s/k)>0$. 
\begin{align}
\sum_{n \le T} r_{k, s}(n) e(n \alpha) &= C_{k, s}\SkqaqPower \, \int_0^T \,  \frac{e(\beta w)}{w^{1-s/k}} \, dw \,  + \, O(T^{(s-1)/k} q^{1/2+\epsilon} (1+T|\beta|)^{1/2}), \nonumber \\
\nukNhatalpha &= C_{k, s} \SkqaqPower \, 
N \, \left(\frac{\sin(\pi N \beta)}{\pi N \beta}\right)^2 \, + \, O(N^{1-1/k} q^{1/2+\epsilon} (1+N|\beta|)^{1/2}) \label{eq:MajorArcEstkOlderApproach}\\
&\ge - \, O_+(N^{1-1/k} \,  q^{1/2+\epsilon} \, (1+N|\beta|)^{1/2}). \label{eq:MajorArcBoundkOlderApproach}
\end{align}
The first two estimates are true for \textit{any} fixed $s \in \NatNos$, and \eqref{eq:MajorArcBoundkOlderApproach} follows from \eqref{eq:MajorArcEstkOlderApproach} for any even $s \in \NatNos$. Unlike Corollary \ref{cor:MinorArcBoundk}, the minor arc estimate also requires an induction on $s$ in the style of \eqref{eq:rskInduct}, but as expected from the remark after Corollary \ref{cor:MinorArcBoundk}, we get the much weaker saving 
$$\nuNhatalpha \,  \ll \,  
N^{1+\delta/k} \, \left(\frac1q \, + \, \frac1{N^{1/k}} \, + \, \frac qN\right)^{\sigma_k},$$
uniformly in all $N \in \NatNos$, and in all $\alpha \in \reals$, $q \in \NatNos$, $a \in \Z$ satisfying $\gcd(a, q)=1$ and $|\alpha - a/q| \le q^{-2}$. The rest of the proof in section \ref{sec:MainThmCompletion} goes through with the \textit{same} 
major and minor arcs, and yields the weaker bound $|A| \ll_{k, \, s, \, \epsilon} \, N^{1-\sigma_k/k+\epsilon}$, but one that is valid for \textit{any} even $s>0$. 

The case $k=3$ is particularly noteworthy in terms of simplicity: Here we don't even need a minor arcs estimate, for given any $\alpha \in \reals$, 
there exist coprime $a, q \in \Z$ satisfying $|\alpha-a/q| \, \le \, 1/(qN^{1/2})$ with $1 \le q \le N^{1/2}$. By a computation  analogous to \eqref{eq:muNHatUnifErrBound}, this yields the lower bound $\nuNhatalpha \, \ge \, - \, O_+(N^{11/12+\epsilon/2})$ uniformly in \textit{all} $\alpha \in [0, 1]$. Hence, we obtain 
\begin{align*}
 0 \, &= \, \int_0^1 \, |\OneAHatalpha|^2 \, \nuNhatalpha \, d\alpha \, = \, \int_0^{1/N^{1+\epsilon/2}} \, |\OneAHatalpha|^2 \, \nuNhatalpha \, d\alpha \, + \, \int_{1/N^{1+\epsilon/2}}^1 \, |\OneAHatalpha|^2 \, \nuNhatalpha \, d\alpha\\
 &\gg |A|^2 \cdot N^{-\epsilon/2} \, - \, O_+(N^{11/12+\epsilon/2} \, |A|), 
\end{align*}
leading to $|A| \, \ll \, N^{11/12+\epsilon}$ for any even $s>0$, and any $A \subset [N]$ satisfying  $(A-A) \cap \mathcal A_{3, s} = \emptyset$. 
\section*{Acknowledgements}
This project arose out of discussions at the Simons Summer Schools and Workshops, which took place at the Renyi Institute from May 11th to June 22nd. The authors would like to thank the Ed\H{o}s Center at the Renyi Institute for their hospitality, and the organizers of the Simons programs for their meticulous hardwork in organizing these events. The authors would also like to thank M\'at\'e Matolcsi and Imre Ruzsa for useful discussions and for their interest in the problem. {The third author was
partially supported from the Simons Foundation grant Grant MPS-TSM-854813 and NFKIH Excellence 15421.} The fourth author acknowledges support from the Charles University Grant PRIMUS/25/SCI/008.
\section*{Statement on AI use}
{The authors used OpenAI’s ChatGPT, specifically GPT-5.5 to  search for references and test parameter choices. The mathematical arguments and the exposition are the authors’ own, and the authors take full responsibility for the contents of the paper.}
\section*{Conflict of interest}
The authors declare no conflict of interest.

\bibliographystyle{amsalpha}
\bibliography{SarkozyksReferences}

@article{Furstenberg77,
AUTHOR={Furstenberg, H.},
TITLE = {Ergodic behavior of diagonal measures and a theorem of \uppercase{S}zemer\'edi on arithmetic progressions},
JOURNAL = {J.~Analyse Math.},
FJOURNAL = {Journal d'Analyse Math\'ematique},
    VOLUME = {31},
      YEAR = {1977},
     PAGES = {204--256},
}

@article{Sarkozy78,
AUTHOR={S{\'a}rk{\"o}zy, A.},
TITLE = {On difference sets of sequences of integers. \uppercase{I}},
JOURNAL = {Acta Math. Acad. Sci. Hungar.},
FJOURNAL = {Acta Mathematica Academiae Scientiarum Hungaricae},
    VOLUME = {31},
    NUMBER = {1--2},
      YEAR = {1978},
     PAGES = {125--149},
}

@misc{GreenSawhney,
AUTHOR={Green, B. and Sawhney, M.},
TITLE = {New bounds for the \uppercase{F}urstenberg--\uppercase{S}\'ark\"ozy theorem},
NOTE = {arXiv preprint},
YEAR = {2024},
}

@misc{IntersectiveGroup,
AUTHOR={Adajar, C. F. E. and Agrawal, R. and Choudhuri, M. R. and Chuah, C. Y. and Fan, S. and Hegde, S. and Lott, A. and Nandakumar, K. and Ponagandla, N. R.},
TITLE = {Extensions of the \uppercase{F}urstenberg--\uppercase{S}\'ark\"ozy theorem via the arithmetic level-$d$ inequality},
NOTE = {arXiv preprint},
YEAR = {2022026},
}

@article{DoyleRice,
AUTHOR={Doyle, J. R. and Rice, A.},
TITLE = {Multivariate polynomial values in difference sets},
JOURNAL = {Discrete Anal.},
FJOURNAL = {Discrete Analysis},
      YEAR = {2021},
     PAGES = {Paper No. 11, 46 pp.},
}

@article{Bour17,
AUTHOR={Bourgain, J.},
TITLE = {On the \uppercase{V}inogradov mean value},
JOURNAL = {Proc.~Steklov Inst.~Math.},
FJOURNAL = {Proceedings of the Steklov Institute of Mathematics},
    VOLUME = {296},
      YEAR = {2017},
     PAGES = {30--40},
}

@article{BV21,
AUTHOR={Br\"{u}dern, J. and Vaughan, R.},
TITLE = {A \uppercase{M}ontgomery–\uppercase{H}ooley Theorem for sums of two cubes},
JOURNAL = {Eur.~J.~Math.},
FJOURNAL = {European Journal of Mathematics},
    VOLUME = {7},
      YEAR = {2021},
     PAGES = {1616--1644},
}

@book{Vau81,
    AUTHOR = {Vaughan, R.},
     TITLE = {The Hardy--Littlewood Method},
 PUBLISHER = {Cambridge Tracts in Mathematics, Cambridge University Press},
      VOLUME = {80},
      YEAR = {1981},
     PAGES = {xii+364 pp.},
}

@article{green,
  title={On S{\'a}rk{\"o}zy’s theorem for shifted primes},
  author={Green, Ben},
  journal={Journal of the American Mathematical Society},
  volume={37},
  number={4},
  pages={1121--1201},
  year={2024}
}

@misc{fanlott,
AUTHOR={Fan, S. and Lott, A.},
TITLE = {The van der \uppercase{C}orput property for sums of two squares},
NOTE = {arXiv preprint},
YEAR = {2026},
}

@article{Kamae,
  author  = {Kamae, T. and Mend{\`e}s France, M.},
  title   = {Van der Corput's difference theorem},
  journal = {Israel Journal of Mathematics},
  volume  = {31},
  number  = {3--4},
  pages   = {335--342},
  year    = {1978},
  doi     = {10.1007/BF02761498},
  url     = {https://link.springer.com/article/10.1007/BF02761498}
}

@book{mont94,
  title     = {Ten Lectures on the Interface Between Analytic Number Theory and Harmonic Analysis},
  author    = {Montgomery, Hugh L.},
  series    = {CBMS Regional Conference Series in Mathematics},
  volume    = {84},
  year      = {1994},
  publisher = {American Mathematical Society},
  address   = {Providence, RI},
  isbn      = {978-0-8218-0737-8}
}

@article{materuzsa,
  title={Difference sets and positive exponential sums I. General properties},
  author={Matolcsi, M{\'a}t{\'e} and Ruzsa, Imre Z},
  journal={Journal of Fourier Analysis and Applications},
  volume={20},
  number={1},
  pages={17--41},
  year={2014},
  publisher={Springer}
}

@article{nin,
  title={Positive exponential sums and odd polynomials},
  author={Nin{\v{c}}evi{\'c}, Marina and Slijep{\v{c}}evi{\'c}, Sini{\v{s}}a},
  journal={Rad Hrvatske akademije znanosti i umjetnosti. Matemati{\v{c}}ke znanosti},
  number={519= 18},
  pages={35--53},
  year={2014},
  publisher={Hrvatska akademija znanosti i umjetnosti}
}

@article{mat21,
  title={Difference sets and positive exponential sums. II: cubic residues in cyclic groups},
  author={Matolcsi, M{\'a}t{\'e} and Ruzsa, Imre Z},
  journal={Proceedings of the Steklov Institute of Mathematics},
  volume={314},
  number={1},
  pages={138--143},
  year={2021},
  publisher={Springer}
}

@incollection{Ruzsa1984,
  author    = {Ruzsa, I. Z.},
  title     = {Connections between the uniform distribution of a sequence and its differences},
  booktitle = {Topics in Classical Number Theory, Vol. I, II (Budapest, 1981)},
  editor    = {Hal{\'a}sz, G.},
  series    = {Colloquia Mathematica Societatis J{\'a}nos Bolyai},
  volume    = {34},
  pages     = {1419--1443},
  publisher = {North-Holland},
  address   = {Amsterdam, New York},
  year      = {1984}
}

@article{Rice19,
AUTHOR={Rice, A.},
TITLE = {A maximal extension of the best-known bounds for the \uppercase{F}urstenberg–\uppercase{S}\'ark\"ozy theorem},
JOURNAL = {Acta Arith.},
FJOURNAL = {Acta Arithmetica},
    VOLUME = {187},
      YEAR = {2019},
     PAGES = {1--41},
}

@article{KW10,
  author  = {Kawada, K. and Wooley, T.},
  title   = {Relations between exceptional sets for additive problems},
  journal = {J.~Lond.~Math.~Soc.},
  series  = {2},
  volume  = {82},
  number  = {2},
  pages   = {437--458},
  year    = {2010},
  doi     = {10.1112/jlms/jdq036}
}

@article{W15,
  author  = {Wooley, T.},
  title   = {Sums of three cubes, {II}},
  journal = {Acta Arith.},
  volume  = {170},
  number  = {1},
  pages   = {73--100},
  year    = {2015},
  doi     = {10.4064/aa170-1-6},
  eprint  = {1502.01944},
  archivePrefix = {arXiv},
  primaryClass = {math.NT}
}

@article{MC84,
  author  = {McCurley, K.~S.},
  title   = {An effective seven cube theorem},
  journal = {J.~Number Theory},
  volume  = {19},
  number  = {2},
  pages   = {176--183},
  year    = {1984},
  doi     = {10.1016/0022-314X(84)90100-8}
}

@article{Lin43,
  author  = {Linnik, Y.~V.},
  title   = {On the representation of large numbers as sums of seven cubes},
  journal = {Recueil Math{\'e}matique. Nouvelle S{\'e}rie},
  volume  = {12(54)},
  number  = {2},
  pages   = {218--224},
  year    = {1943}
}
\bigskip

\end{document}